\documentclass[letterpaper, 10 pt, conference]{ieeeconf}  

\usepackage[noadjust]{cite}
\usepackage{textcomp}

\usepackage{times}

\usepackage{microtype}

\usepackage[hidelinks]{hyperref}

\usepackage{amsfonts}

\usepackage{amsmath}
\usepackage{amsthm}
\usepackage{amssymb}
\usepackage{extarrows}
\usepackage{enumerate}
\usepackage{graphicx}
\usepackage{tikz}
\usepackage{booktabs}
\usepackage{verbatim}
\usepackage{bm}
\usepackage{bbm}
\usepackage{epstopdf}

\usepackage{mathtools}

\usepackage{url}

\usepackage[off,crop=off]{auto-pst-pdf}

\newcommand{\RE}{\mathbb{R}}
\newcommand{\REp}{\mathbb{R}_{\geq0}}

\newcommand{\fin}{\Phi}
\newcommand{\run}{\Psi}
\newcommand{\hfin}{\hat{\fin}}
\newcommand{\hrun}{\hat{\run}}
\newcommand{\I}{Z}
\newcommand{\w}{r}

\newcommand{\calH}{\mathcal{H}}

\newcommand{\calO}{\mathcal{O}}

\usepackage{tikz}

\usepackage{color}

\def\R{{\mathbb R}}

\def\calH{{\mathcal H}}

\DeclareMathOperator{\diag}{diag}
\DeclareMathOperator{\rowspan}{rs}

\DeclareMathOperator{\supp}{supp}
\newcommand{\lb}{\underline{A}}
\newcommand{\ub}{\overline{A}}
\newcommand{\lbb}{\underline{R}}
\newcommand{\ubb}{\overline{R}}

\newcommand{\B}{R}
\newcommand{\rwlL}{W}

\newtheorem{example}{Example}
\newtheorem{lemma}{Lemma}
\newtheorem{proposition}{Proposition}
\newtheorem{theorem}{Theorem}
\newtheorem{remark}{Remark}
\newtheorem{definition}{Definition}

\begin{document}

\title{\LARGE \bf
Proper Lumping for Positive Bilinear Control Systems
}

\author{A. Jim\'{e}nez-Pastor$^{1}$, D. Toller$^{1}$, M. Tribastone$^{2}$, M. Tschaikowski$^{1}$ and A. Vandin$^{3,4}$ \vspace{1em}\\
\small{$^{1}$Aalborg University, Denmark,\ \ $^{2}$IMT Lucca, Italy,\ \ $^{3}$Sant'Anna Pisa, Italy,\ \ $^{4}$DTU Compute, Denmark}
}


\maketitle

\begin{abstract}
Positive systems 
naturally arise in situations where the model tracks physical quantities. Although the linear case is well understood, analysis and controller design for nonlinear positive systems remain challenging. Model reduction methods can help tame this problem. Here we propose a notion of model reduction for a class of positive bilinear systems with (bounded) matrix and exogenous controls. Our reduction, called proper positive lumping, aggregates the original system such that states of the corresponding reduced model represent non-negative linear combinations of original state variables. We prove a characterization result showing that the reductions by proper positive lumping are exactly those preserving the optimality of a suitable class of value functions. Moreover, we provide an efficient polynomial-time algorithm for the computation of the minimal lumping. We numerically evaluate our approach by applying it to a number of benchmark case studies.
\end{abstract}


\section{Introduction}

Positive systems find applications in many domains where models describe the evolution of variables that live in the non-negative orthant~\cite{farina2000positive,smith2008monotone}, e.g., to track physical quantities~\cite{shorten2006positive,haddad2010nonnegative,kaczorek2011positive}.
In this paper we are concerned with controlled positive systems in the form $\partial_t x(t) = A(t)x(t) + B u(t)$ with an output map $z(t) = O x(t)$, where $A(\cdot)$ and $u(\cdot)$ are controls subject to compact domains.
This forms a subset of well-known bilinear control systems~\cite{1000269} with exogenous controls. Although positive linear systems have received considerable attention in the literature (e.g.,~\cite{8618689} and references therein), fewer results are available for bilinear control systems. Despite some advances~\cite{AMATO20011459,4100876,WEISS2021109358}, problems like stability and controller design still remain challenging.


Model reduction methods can simplify those problems by projecting the original model onto a lower-dimensional one with theoretical guarantees of a relationship between the two. Model reduction has a long-standing tradition
across many branches of science (e.g.,~\cite{1098900,antoulas2005approximation,Snowden:2017aa}). Here we are concerned with a specific class of methods known as \emph{proper lumping} (e.g.,~\cite{okino1998}): the projection is performed by identifying a positive orthogonal projection of the original state variables, and the reduced model tracks the state evolution of the projection. This ensures physically intelligible reductions~\cite{PNAScttv,DBLP:conf/qest/CardelliTTV18,DBLP:journals/tcs/CardelliTTV19a} that preserve positiveness~\cite{Snowden:2017aa}.

The vast majority of lumping algorithms share two main features: i) they do not take controls into account; ii) they consider reductions induced by partitions, where each partition block represents the \emph{sum} of the original variables in that block~\cite{Derisavi2003309,dokoumetzidis:40,PNAScttv}. In this paper, we develop a lumping method for positive linear systems with matrix and vector controls where each block represents a \emph{non-negative linear combination} of variables; for these two reasons, this method is called \emph{proper positive lumping}. 

We define (constrained) proper lumping for a positive control system as a reduction that preserves a given output matrix $O$ (known as the constraint) and that is a proper lumping for the extremal matrices $\lb$ and $\ub$ describing the matrix domain $A(\cdot) \in [\lb;\ub]$. We show that the reduced model preserves optimality with respect to the original one under a suitable class of costs. In particular, we provide a result that allows the reconstruction of the original optimal controls from the reduced ones, thus circumventing the original optimal control problem altogether. Importantly, we also give a characterization result: we prove that a reduction is optimality-preserving for the control system if and only if it is a proper lumping of its extremal matrices $\lb$ and $\ub$.

Hence, applying out proper lumping technique, we can compute the minimal control system necessary to obtain the output map $z(t)$ and perform optimal computations on the reduced system to, then, reconstruct the optimal controls on the original system. 

Another important contribution of this paper is of algorithmic nature. We show that the minimal proper lumping, i.e., the one projecting onto the smallest space, can be computed in polynomial time with respect to the dimensionality of the model. The algorithm computes a common minimal proper lumping of the extremal matrices $\lb$ and $\ub$ by building on \emph{CLUE}~\cite{10.1093/bioinformatics/btab258}, a recent method for lumping ordinary differential equations with polynomial right-hand sides (without controls).

Using CLUE, we extensively test proper lumping on over 1000 models of dynamics on networks taken from a public repository. We show that the notion can be useful in practice as it can reduce at least half of the examined models, with at least a 65\% reduction ratio for over 30\% of the cases.

\emph{Related work.} Our approach generalizes a recent work that considers uncertain continuous-time Markov chains~\cite{uctmcTAC}, which are a special case of positive linear systems with interval uncertainties. In addition, the lumping algorithm in~\cite{uctmcTAC} only considers blocks that represent sums of variables, unlike our more general case of non-negative linear combinations.

In the field of model reduction, the notions most close to us are those of constrained lumping~\cite{10.1093/bioinformatics/btab258,LI1994343}
and consistent abstraction, also known as bisimulation~\cite{DBLP:journals/tac/PappasLS00,bisimulation_lin_sys_Schaft,DBLP:journals/tac/PappasS02,DBLP:journals/automatica/TabuadaP03}.
Our notion complements those approaches. Indeed, while lumping has been considered for uncontrolled dynamical systems only, the computational aspects of bisimulation have been investigated in full only in the case of additive controls~\cite{DBLP:journals/tac/PappasLS00}. The nonlinear counterparts~\cite{DBLP:journals/tac/PappasS02,bisimulation_lin_sys_Schaft}, instead, provide a mathematical characterization of the bisimulation quotient (via tangential spaces and manifold intersections) but leave the computational aspects unaddressed. Less closely related are~\cite{antoulas2005approximation,MehrmannStykelMOR,10.1007/978-3-642-02894-6_13,LI20111504} or~\cite{Ishizaki2015,Sandberg2008} as these either do not allow for matrix controls or study network properties.

\section{Preliminaries}

\paragraph*{Notation} We denote by $A$ a real $n \times n$ matrix; instead, $L$ usually denotes a $k \times n$ matrix with full row rank, whereas $L^+ = L^T (L L^T)^{-1}$ denotes its generalized right inverse, satisfying $L L^+ = I_k$ (the $k\times k$ identity matrix). 
The row space of $L$ is denoted by $\rowspan(L)$, 
while the support of a vector $v$ is $\supp (v) = \{ i \mid v_i \neq 0 \}$. A matrix function $A = A(\cdot)$ is called admissible if it is piecewise continuous with right limits and it satisfies  $A(\cdot) \in [\lb ; \ub]$ with respect to matrix bounds $\lb \leq \ub$, where $\leq$ is meant component-wise. The technical assumption on right limits is necessary to establish later our characterization results. We often denote by $x^{(A,u)}$ the solution of $\partial_t x(t) = A(t) \cdot x(t) + B \cdot u(t)$.

We begin by reviewing  constrained lumping ~\cite{10.1093/bioinformatics/btab258,LI1994343}.

\begin{definition}[Constrained Lumping]\label{new:def:lumping}
Fix a linear system with state matrix $A \in \RE^{n \times n}$ and output matrix $O \in \RE^{l \times n}$: 
\begin{align}\label{linear:odes}
\partial_t x(t) & = A \cdot x(t), &
		z(t) & = O \cdot x(t).
\end{align}

A matrix $L \in \RE^{k \times n}$ with rank $k \leq n$ is called a constrained lumping of \eqref{linear:odes} if $\rowspan(O) \subseteq \rowspan(L)$ and $\rowspan(LA) \subseteq \rowspan(L)$.
%
\end{definition}
In the above setting, we also say that $L$ is a (constrained) lumping of $A$ with respect to $O$.

\begin{example}\label{running:example}
Consider the linear system
\begin{align*}
\partial_t x(t) & = A \cdot x(t), & z(t) & = O \cdot x(t),
\end{align*}
where
$
A =
\begin{pmatrix}
0 & 2 & 0\\
1 & 1 & 1\\
0 & 0 & 1
\end{pmatrix},
$ and 
$O =
\begin{pmatrix}
0 & 0 & 1
\end{pmatrix}.$
%
%
Then the matrix $L = \begin{pmatrix}
1 & 2 & 0\\
0 & 0 & 1
\end{pmatrix}$
is a constrained lumping.
\end{example}

\begin{definition}[Aggregation and Reduced System]
Fix a matrix $L \in \RE^{k \times n}$ with rank $k \leq n$.
\begin{itemize}
\item The function $\R^n \to \R^k$, mapping $x \mapsto Lx$, induced by $L$ is called the aggregation (map), and $y=Lx$ are called the aggregated variables.
\item When $\rowspan(O) \subseteq \rowspan(L)$, setting $\B = L A L^+$, we call 
\begin{align}\label{reduced:linear:odes}
\partial_t y(t) & = \B \cdot y(t) , & z(t) & = O L^+ \cdot y(t)
\end{align}
the reduced system of~\eqref{linear:odes}. 
\end{itemize}
\end{definition}

The condition of lumping was recently proved~\cite{10.1093/bioinformatics/btab258} to be equivalent to the following property, motivating our interest in them and the definition of reduced system.
\begin{proposition}
\label{preservation:traj:0}
If $L \in \RE^{k \times n}$ is of rank $k$ with $\rowspan(O) \subseteq \rowspan(L)$, then $L$ is a lumping of~\eqref{linear:odes} if and only if, for any solution $x$ of~\eqref{linear:odes}, $y = L x$ is a solution of~\eqref{reduced:linear:odes}.
\end{proposition}

\begin{remark}
The choice of the output matrix for the reduced system in \eqref{reduced:linear:odes} is motivated by the fact that $\rowspan(O) \subseteq \rowspan(L)$ can be shown to be equivalent to
$O = O L^+ L$, and then 
$$z = O x = O L^+ L x = O L^+ y,$$ so the original system and the reduced one have the same output values $z$.
\end{remark}

\begin{example}\label{running:example:reduced}
Continuing Example \ref{running:example}, one has 
\[
L^+ = \begin{pmatrix}
\frac{1}{5} & 0\\
\frac{2}{5} & 0\\
0 & 1
\end{pmatrix},
\
L A L^+ = \begin{pmatrix}
2 & 2\\
0 & 1
\end{pmatrix},
\text{ and }
O L^+ = \begin{pmatrix}
0 & 1
\end{pmatrix}.
\]
The aggregated variables are $\begin{pmatrix}
y_1 \\
y_2
\end{pmatrix} =
\begin{pmatrix}
x_1 + 2 x_2 \\
x_3
\end{pmatrix}$ and the reduced system is
\begin{align}\label{ex:reduced}
\begin{pmatrix}
\partial_t y_1\\
\partial_t y_2
\end{pmatrix}
& =
\begin{pmatrix}
2 & 2\\
0 & 1
\end{pmatrix}
\begin{pmatrix}
y_1\\
y_2
\end{pmatrix},
&
z = \begin{pmatrix}
0 & 1
\end{pmatrix}\begin{pmatrix}
y_1\\
y_2
\end{pmatrix}.
\end{align}
\end{example}

Given an output $O$, a constrained lumping $L \in \RE^{k \times n}$ with minimal row space $\rowspan(L)$ exists and can be computed in polynomial time~\cite{10.1093/bioinformatics/btab258}. We will make use of this result in Section~\ref{section:evaluation}.

The matrix $L$ in Example \ref{running:example} has non-negative entries, and its rows have pairwise disjoint supports. Matrices with such properties have been studied in the literature~\cite{okino1998}; they are fundamental to our work, so we introduce them in the following definition.
\begin{definition}[Proper Matrix and Proper Lumping]
We say a matrix $L \in \RE_{\geq 0}^{k \times n}$ is \emph{proper} if it has rank $k \leq n$, and its rows have pairwise disjoint supports.

If $L$ is proper, and it is a constrained lumping of \eqref{linear:odes}, we call $L$ a \emph{constrained proper lumping} of \eqref{linear:odes} (or simply \emph{proper lumping} when this creates no ambiguity). 
\end{definition}

For a matrix $L \in \RE^{k \times n}$ with rank $k \leq n$, the aggregation map $\R^n \to \R^k$ is surjective, as $L$ has full row rank by definition, and transforms the $n$-dimensional problem \eqref{linear:odes} to the $k$-dimensional one \eqref{reduced:linear:odes}. If $L$ has any zero columns, then $L$ essentially deletes the corresponding state variables, as they do not appear in the aggregation $L x$; obviously $L$ induces an aggregation of the remaining variables. Each of the rows of $L$ represents a state variable for the reduced system, obtained as the linear combination of the original state variables using that row of coefficients.
The additional algebraic property of $L$ to be proper translates to the fact that such linear combinations are taken over disjoint subsets of variables, so each original variable influences at most one variable of the reduced system.

\section{Lumping Positive Control Systems}

We begin by introducing the notion of 
Positive Controlled Systems (PCS).
Following the literature~\cite{WEISS2021109358}, the system \eqref{linear:odes} is called positive if it satisfies the following: for every non-negative initial condition $x[0] \in \REp^n$, the solution $x(\cdot)$ of \eqref{linear:odes} remains non-negative, i.e. $x(t) \in \REp^n$ for every $t\geq 0$. It is well known that this holds if and only if $A_{i,j} \geq 0$ for every $1 \leq i \neq j \leq n$, and a matrix with this property is called Metzler matrix. Adhering to this setting, we thus consider in the following Metzler matrices, non-negative input matrices $B$, and non-negative control inputs $u$.

\begin{definition}[Positive Controlled System, PCS]\label{definition:uncertain:linear:system}
A PCS, denoted by $([\lb;\ub],O,B,U)$,  is given by
\begin{align}\label{ULODE:uncertain:linear:odes:0}
\partial_t x(t) & = A(t) \cdot x(t) + B \cdot u(t), &
z(t) & = O \cdot x(t)
\end{align}
where $O \in \RE^{l \times n}$ is the output matrix, $B \in \RE^{n \times m}_{\geq0}$, and
\begin{itemize}
    \item $\lb, \ub \in \RE^{n \times n}$ are Metzler matrices with $\lb \leq \ub$;
    \item $U \subseteq \RE_{\geq0}^m$ is  compact and contains the origin;
    \item $A(\cdot) \in [\lb;\ub]$ and $u(\cdot) \in U$ are admissible.
\end{itemize}
We denote by $x[0] \in \REp^n$ the initial condition of the PCS.
\end{definition}

As $\lb$ is Metzler, all admissible matrices $A(\cdot) \geq \lb$ are Metzler too, and so given $x[0] \in \REp^n$ all solutions to \eqref{ULODE:uncertain:linear:odes:0} remain in $\REp^n$.
We remark that PCSs constitute a family of bilinear control systems~\cite{1000269} with exogenous linear controls, i.e., both the matrix $A(t)$ and $u(t)$ are control variables. Being interested in the optimal control of a PCS, we next introduce the value function associated to it~\cite{Liberzon}. 

\begin{definition}[Value Function]\label{def:values}
For time $\tau \in \R_{\geq0}$ and vector $x[0] \in \REp^n$ describing the initial condition of~\eqref{ULODE:uncertain:linear:odes:0}, we define
\begin{itemize}
    \item for any $A(\cdot) \in [\lb ; \ub]$ and $u(\cdot) \in U$, the cost as
    \begin{multline*}
    J_{(A,u)}(x[0],\tau) = \fin(x^{(A,u)}(\tau),u(\tau)) \\
    + \int_0^{\tau} \run(t,x^{(A,u)}(t),u(t)) dt,
    \end{multline*}
    where $x^{(A,u)}$ is the solution of~\eqref{ULODE:uncertain:linear:odes:0} with given $A$ and $u$, the continuous function $\run: \R \times \R^n \times \R^m \to \R$ is called running cost, and the continuous function $\fin: \R^n \times \R^m \to \R$ is called final cost.    
    \item the optimal values as
    \begin{gather*}
    \hspace{-5pt}
    V^{\inf} (x[0], \tau) = \inf \{ J_{(A,u)}( x[0], \tau ) \mid A \in [\lb;\ub], u \in U \},\\
    \hspace{-10pt}
    V^{\sup} (x[0], \tau) = \sup \{ J_{(A,u)}( x[0], \tau ) \mid A \in [\lb;\ub], u \in U \}.
    \end{gather*}
\end{itemize}
\end{definition}

The compactness of $[\lb;\ub]$ and $U$ ensure the well-definedness of the value functions. While we consider costs $\Psi, \Phi$ that do not depend directly on $A$, we argue that this is a mild restriction because the control set $[\lb;\ub]$ is bounded. We note in particular that we can encode reachability, bounded affine controls $(A,u)$ and costs quadratic in state $x$ and control $u$, as studied in the linear quadratic regulator~\cite{Liberzon}.

We next introduce the main concept of the present work, by lifting the classical notion of constrained lumping to a PCS. Due to the non-negativity of the solutions of a PCS, we consider aggregation maps given by non-negative matrices.

\begin{definition}[Constrained Proper Positive Lumping]\label{def:lumping+}
Fix a PCS $([\lb;\ub],O,B,U)$.
\begin{itemize}
    \item We call $L$ a \emph{constrained proper positive lumping of the PCS} 
    if $L \in \RE_{\geq 0}^{k \times n}$ is proper and is a constrained lumping of $\lb$, $\ub$ with respect to $O$.
    \item The reduced PCS is $([ L \lb L^+ ; L \ub L^+ ],OL^+,LB,U)$, with initial condition
   $y[0] = L x[0]\in \R_{\geq 0}^{k}$.
\end{itemize}
\end{definition}

Often, we simply call proper lumping a matrix $L$ as above.

\begin{example}\label{example:uncertain:case}
Consider the PCS that arises by perturbing $A$
from Example~\ref{running:example} by a factor of $10 \% $, yielding $\lb = 0.9 A$ and $\ub = 1.1 A$.
Let $O = \begin{pmatrix}
0 & 0 & 1
\end{pmatrix}$, and consider the PCS $([\lb ; \ub], O, 0, \{0\} )$. 
Since $L$ from Example~\ref{running:example}
is a constrained lumping for both $\lb$ and $\ub$, it is constrained proper lumping for the above PCS. The reduced PCS is given by
\begin{align*}
L \lb L^+ & = 0.9 \begin{pmatrix}
2 & 2 \\
0 & 1
\end{pmatrix} &
L \ub L^+ & = 1.1 \begin{pmatrix}
2 & 2 \\
0 & 1
\end{pmatrix} \\
O L^+ & = \begin{pmatrix}
0 & 1
\end{pmatrix} &   LB & = 0.
\end{align*}
\end{example}

Next we introduce the natural candidates for cost functions and optimal values of the reduced PCS. To this end, note that given $L \in \RE^{k \times n}$ one has $L^+ \in \RE^{n \times k}$, so its associated linear map is $L^+: \RE^k \to \RE^n$.

\begin{definition}[Reduced optimal values and cost preservation]
Fix a PCS and a matrix $L \in \RE^{k \times n}$ of rank $k \leq n$.
\begin{itemize}
\item
Final and running costs $\fin$, $\run$ are called \emph{$L$-invariant} if $\fin (x,u) = \fin (x',u)$ and $\run (t,x,u) = \run (t,x',u)$ whenever $x, x' \in \R ^n$ satisfy $L x = L x'$. In this case, we consider the \emph{reduced final and running costs} $\hfin : \R^k \times \R^m \to \R$ and $\hrun : \R \times \R^k \times \R^m \to \R$ defined by
\begin{align*}
\hfin(y,u) = \fin (L^+ y,u) \text{ \ and \ }
\hrun(t,y,u) = \run(t,L^+y,u).
\end{align*}
Using them, the values for the reduced PCS 
are as in Definition~\ref{def:values}, and we denote them by $\hat{V}^{\inf}$ and $\hat{V}^{\sup}$.
\item
We say that $L$ \emph{preserves the costs} 
if for any $x[0] \in \REp^n$, it holds that $y[0] = L x[0]$ satisfies
    \begin{align*}
    V^{\inf} (x[0], \tau) & = \hat{V}^{\inf} (y[0], \tau),\\
    V^{\sup} (x[0], \tau) & = \hat{V}^{\sup} (y[0], \tau)
    \end{align*}
for every $\tau \in \REp$ and for every final cost $\fin$ and running cost $\run$ that are $L$-invariant.
\end{itemize}
\end{definition}
We remark that the above definitions are well-posed as the matrix $L$ has full row rank, and both $\hfin$ and $\hrun$ are continuous. 


\section{Main results}

We show first that a constrained proper lumping of a PCS preserves the family of its solutions. As a consequence, it also preserves its costs.

\begin{theorem}[Cost Preservation]\label{thm:coincidence:trajectories}
Assume that $L \in \REp^{k \times n}$ is a constrained proper lumping of the PCS $([\lb;\ub],O,B,U)$. Then for any initial condition $x[0] \in \REp^n$, and any $u(\cdot) \in U$:
\begin{itemize}
 \item for any $A(\cdot) \in [\lb;\ub]$, 
 there is $\B(\cdot) \in [ L \lb L^+; L \ub L^+ ]$ such that $L x^{(A,u)} = y^{(\B,u)}$;
 \item for any $\B(\cdot) \in [ L \lb L^+; L \ub L^+ ]$, 
 there is $A(\cdot) \in [\lb;\ub]$ such that $L x^{(A,u)} = y^{(\B,u)}$.
 \end{itemize}
In particular, $L$ preserves the costs of the PCS.
\end{theorem}

In the next result we give the converse to Theorem \ref{thm:coincidence:trajectories}, thus providing a complete characterization of a constrained proper lumping in terms of preservation of costs.

\begin{theorem}[Cost Characterization]\label{new:thm:coincidence:cost:functionals}\label{thm:coincidence:cost:functionals}
Let $L \in \R_{\geq 0}^{k \times n}$ be a proper matrix. If $\rowspan(O) \subseteq \rowspan(L)$ and $L$ preserves the costs of a PCS $([\lb;\ub],O,B,U)$, then $L$ is a constrained proper lumping of the PCS.
\end{theorem}

\subsection{Control reconstruction and proofs of Theorem~\ref{thm:coincidence:trajectories} and~\ref{thm:coincidence:cost:functionals}}

In this section we give some insight about the results previously presented.

\begin{definition}\label{induced:partition}
Given a proper matrix $L \in \RE_{\geq 0}^{k \times n}$, 
for every integer $1\leq \w \leq k$, let us introduce the disjoint subsets of indexes (called blocks)
$H_{\w} = \supp( \text{$\w$-th row of } L ),
$ 
and $H_0 = \{ i \mid \text{ the $i$-th column of $L$ is zero}\}$. Finally, for every $1 \leq i \leq n$, let us denote by $[i]$ the unique set of the above such that $i \in H_{[i]}$.
\end{definition}
In the above terminology, note that $H_0$ is the complement of the disjoint union $\bigcup_{\w=1}^k H_{\w}$. A proper lumping deletes the state variables indexed by $H_0$ and  aggregates the remaining variables indexed by $\bigcup_{\w=1}^k H_{\w}$.
Then for every $x \in \R^n$ the aggregation $Lx$ has the form
\[
Lx = (\sum_{i \in H_1} \lambda_i x_i, \ldots, \sum_{i \in H_k} \lambda_i x_i)^T
\]
for positive weights $\lambda_i$.

Proposition~\ref{lumping+:as:cdp} below formalizes this discussion, by ensuring that $L$ can be written as a product $CD$, where $C$ encodes the blocks, while $D$ captures the weight associated by $L$ to each variable. Moreover, we identify the matrices of this form that are lumpings.

\begin{proposition}\label{lumping+:as:cdp}
If $L \in \RE_{\geq 0}^{k \times n}$ is a proper matrix, then it can be written as $L = CD$, where  
\begin{itemize}
\item $C \in \{0,1\}^{k \times n}$ is given by $C_{\w,i} = 1$ if $L_{\w,i} > 0$ and $C_{\w,i} = 0$ otherwise, for any $1 \leq \w \leq k$ and $1 \leq i \leq n$;
\item 
$D = \diag (\lambda_i)_{1 \leq i \leq n} \in \R_{> 0}^{n \times n}$ is the diagonal matrix whose $(i,i)$-th entry $\lambda_i$ is one of the positive entries of $L$ if 
$i \notin H_0$, while $\lambda_i = 1$ otherwise.
\end{itemize}

Given matrices $A$ and $O$, $L$ is a constrained proper lumping of $A$ (wrt $O$) if and only if $C$ is a constrained proper lumping of $D A D^{-1}$ (wrt $O D^{-1}$).
\end{proposition}

A proper matrix $L$ and the matrix $C$ introduced in Proposition~\ref{lumping+:as:cdp} define the same blocks, introduced as in Definition \ref{induced:partition}. Indeed, for every $r \leq k$, the $r$-th row of $L$ and the $r$-th row of $C$ have the same support. 
So there is no ambiguity if we write $[i]$ for the block containing an index $i$.

The next results can be proven by generalizing the proof for uncertain continuous-time Markov chains (UCTMCs) from~\cite{uctmcTAC} and by invoking Proposition~\ref{lumping+:as:cdp}. According to~\cite{uctmcTAC}, a UCTMC is a continuous-time Markov chain whose transition rates from state $j$ into state $i$ are subject to the constraint $A_{i,j}(\cdot) \in [\lb_{i,j} ; \ub_{i,j}]$.

Moreover, while~\cite{uctmcTAC} confined itself to lumpings $C \in \{0,1\}^{k \times n}$ that encode blocks, Proposition~\ref{lumping+:as:cdp} allows us to consider proper lumpings $L$ by expressing these as suitably stretched versions of aforementioned matrices $C$, i.e., $L = CD$.


\begin{proposition}[Cost Preservation -- special case]\label{generalized:UCRN}
Fix a PCS $([\lb;\ub],O,B,U)$, and a proper matrix $C \in \{0,1\}^{k \times n}$. 
Then the following conditions are equivalent:
\begin{enumerate}
\item $C$ is a proper lumping of both $\lb$ and $\ub$;
\item $C$ preserves the costs of $([\lb;\ub],O,B,U)$;
\item for any initial condition $x[0] \in \RE_{\geq0}$ and $u(\cdot) \in U$:
\begin{itemize}
    \item[i)] Any $A(\cdot) \in [\lb;\ub]$ has an $\B(\cdot) \in [ C \lb C^+; C \ub C^+ ]$ such that $C x^{(A,u)} = y^{ ( \B,u ) }$,
    \item[ii)] Any $\B(\cdot) \in [ C \lb C^+; C \ub C^+ ]$ has an $A(\cdot) = ( A_{i,j}(\cdot) )_{i,j}\in [\lb;\ub]$, defined as follows, such that $C x^{(A,u)} = y^{(\B,u)}$ and $C$ is a lumping for $A$. Let $\lbb = C \lb C^+$, and
    for $i,j$ with $[i], [j] \neq 0$, we define
    \begin{align*}
    A_{i,j}(t) & = \displaystyle{ \lb_{i,j} + \frac{ (\ub_{i,j} - \lb_{i,j}) ( \B_{[i],[j]}(t) - \lbb_{[i],[j]})}{ \sum_{k\in [i]} ( \ub_{k,j} - \lb_{k,j} ) }}
    \end{align*}
    provided the division is well-defined; in all other cases, we set $A_{i,j}(t) = \lb_{i,j}$.
\end{itemize}
\end{enumerate}
\end{proposition}

The proofs of Theorems \ref{thm:coincidence:trajectories} and~\ref{thm:coincidence:cost:functionals} consist in reducing the general case to the case considered in Proposition~\ref{generalized:UCRN}. 
Additionally, here follows separately a more detailed statement about how to recover optimal controls for the original system (from ones of the reduced system), providing an explicit formula for them.

\begin{proposition}[Control Reconstruction]\label{prop:reconstruction:general}
Assume that $L$ is a proper positive lumping of the PCS $([\lb;\ub],O,B,U)$.
Write $L = CD$ according to Proposition~\ref{lumping+:as:cdp}, so that $C$ is a lumping of both $\underline{M} = D \lb D^{-1}$ and $\overline{M} = D \ub D^{-1}$, and let $\lbb = L \lb L^+$. 

Then, for any $\B(\cdot) \in [ L \lb L^+; L \ub L^+ ]$, we can define $A(\cdot) = ( A_{i,j}(\cdot) )_{i,j}\in [\lb;\ub]$ as follows. If $[i], [j] \neq 0$, set
\begin{multline*}
\hspace{-10pt}
A_{i,j}(t) = \lambda_{i}^{-1} \lambda_j
\Big(
\underline{M}_{i,j} + \frac{ ( \overline{M}_{i,j} - \underline{M}_{i,j} ) ( \B_{[i],[j]}(t) - \lbb_{[i],[j]} ) }{ \sum_{k\in [i]} ( \overline{M}_{k,j} - \underline{M}_{k,j} ) }
\Big)
%
%
\end{multline*}
provided the division is well-defined; in all other cases, we set $A_{i,j}(t) = \lambda_{i}^{-1} \lambda_j \underline{M}_{i,j}$.

Then $L$ is a lumping for $A(\cdot)$ and $\B(\cdot) = L A(\cdot) L^+$, so $L x^{(A,u)} = y^{(\B,u)}$ for any $x[0] \in \REp^n$ and any $u(\cdot) \in U$.
\end{proposition}

\section{Computation of constrained proper lumping}\label{sec:existence:lumping+}

This section studies the computation of PCS lumpings. We recall that 
finding a PCS lumping amounts to finding a proper lumping of the linear dynamical systems $\partial_t x = \lb x(t)$ and $\partial_t x = \ub x(t)$, constrained with respect to an output matrix $O$. 
This can be achieved by making use of the efficient algorithm CLUE from~\cite{10.1093/bioinformatics/btab258} which computes the minimal constrained lumping $L$ for a system of polynomial differential equations
\begin{align}\label{polynomial:odes}
\partial_t x(t) & = p(x(t)), &
		z(t) & = O \cdot x(t).
\end{align}

Indeed, as shown in~\cite{10.1093/bioinformatics/btab258}, $L$ is a constrained lumping of \eqref{polynomial:odes} if and only if $L$ is a constrained lumping of a certain family of linear dynamical systems that arise from the total derivative of $p$. Feeding CLUE with $\lb$ and $\ub$ instead those linear systems yields then the following.

\begin{proposition}[\cite{10.1093/bioinformatics/btab258}]\label{prop:alg:clue}
For a PCS $([\lb;\ub],O,B,U)$:
\begin{itemize}
    \item 
    Applying CLUE to $\lb$, $\ub$ gives a constrained lumping $L \in \RE^{k \times n}$ of both $\lb$ and $\ub$ (wrt $O$), 
    and with row space dimension $k$;
    \item such $k$ is minimal satisfying all the above properties;
    \item the matrix $L$ is in reduced echelon-form;
    \item the computation of $L$ requires at most $\calO(kn(e+k))$ steps, where $e$ is the total number of non-zero entries in $\lb$, $\ub$ and $O$.
\end{itemize}
\end{proposition}


\begin{example}\label{example:lumping:produced:by:clue}
For the PCS in Example~\ref{example:uncertain:case}, 
CLUE produces the lumping
$
L =
\begin{pmatrix}
1 & 2 & 0 \\
0 & 0 & 1
\end{pmatrix}
$.
All entries of $L$ are non-negative and the supports of its rows are disjoint, hence $L$ is a constrained proper positive lumping of the PCS.
\end{example}

In general, $L$ from Proposition~\ref{prop:alg:clue} may have negative entries, hence it may not be a constrained proper 
lumping. 
The following result certifies that $L$ is the only candidate to check for being a constrained proper 
lumping with minimum row dimension.

\begin{theorem}\label{Clue:for:m:and:M}
Fix a PCS $([\lb;\ub],O,B,U)$, and let $L \in \RE^{k \times n}$ be the result of CLUE as in Proposition~\ref{prop:alg:clue} for $\lb$ and $\ub$ with output map $O$. If $L$ is not proper, then there is no constrained proper lumping with row space dimension $k$.
\end{theorem}

The above result ensures that, in order to compute a constrained proper 
lumping, we can apply the algorithm from~\cite{10.1093/bioinformatics/btab258} to get a matrix $L$. If $L$ is a constrained proper 
lumping for $[\lb;\ub]$, we terminate; otherwise, we are guaranteed that there is no proper lumping with the same row space of $L$.

\begin{remark}
In all models in Section~\ref{section:evaluation}, if the matrix $L$ from Proposition~\ref{prop:alg:clue} has rows with disjoint supports, then it also has non-negative entries, and so it is proper.
\end{remark}

\section{Experimental Evaluation}\label{section:evaluation}

\newcommand{\totalmodels}{1352}
\newcommand{\totalcases}{110808}
\newcommand{\reductioncases}{67371}\newcommand{\reductioncasesP}{60.8\%}
\newcommand{\norwlcases}{39318}\newcommand{\norwlcasesP}{35.5\%}
\newcommand{\noreduction}{2495}\newcommand{\noreductionP}{2.2\%}
\newcommand{\timeoutcases}{1624}\newcommand{\timeoutcasesP}{1.5\%}
\newcommand{\timeout}{10 minutes}
\newcommand{\reponame}{cpl}

We implement PCS proper lumping based on CLUE~\cite{10.1093/bioinformatics/btab258}. The implementation can be found in the \emph{\reponame} branch (see \url{https://github.com/Antonio-JP/CLUE/tree/\reponame}). All examples are available in the \emph{tests/examples/\reponame} folder in the aforementioned repository. This repository also include the scripts necessary to run all the experiments described in this article. Reported experiments refer to a common laptop (Lenovo with 2.8GHz i7 processor) with 23 GB RAM and not exploiting parallelization.

\textbf{Set up.} The models have been taken from the repository of (positive) weighted networks in \url{https://networks.skewed.de/}. We have considered networks in the repository up to 100000 vertices. For each of them we have considered a set of PCS in the form $([0.9A; 1.1A], O_i, 0,\{0\})$ where $A$ is the weighted adjacency matrix of the network and $O_i = e_i^T$ for all $i$ such that the $i$-th row of $A$ is not zero (i.e., the variables that are not constant in the system). Similarly to~\cite{uctmcTAC}, this setting creates interval uncertainties of $\pm 10\%$ around the original values of $A$.

\textbf{Results.} We summarize the results across all matrices $L$ computed by Proposition~\ref{prop:alg:clue} on the considered PCS as follows:
\begin{itemize}
    \item \emph{Proper lumping}: the obtained lumping is a constrained proper 
    lumping with respect to the given $O_i$ and it reduces the size of the original system.
    \item \emph{General lumping}: the obtained lumping is not proper.
    \item \emph{No reduction}: the obtained lumping does not provide any reduction, i.e., the lumping is the identity matrix.
    \item \emph{Timeout examples}: in all experiments, we imposed an arbitrary timeout of \timeout. With this, more than 96\% of the models could be processed. 
\end{itemize}

\begin{table*}[ht]
    \centering
    \begin{tabular}{cccccc}
            \toprule
            \textbf{No. of PCS (no. of models)} & \textbf{Proper} & \textbf{General} &  \textbf{No reduction}  & \textbf{Timeout} \\
            \midrule
            \totalcases~(\totalmodels) &
            \reductioncases~(\reductioncasesP) &
            \norwlcases~(\norwlcasesP) &
            \noreduction~(\noreductionP) &
            \timeoutcases~(\timeoutcasesP)\\
            \bottomrule
        \end{tabular}
        \caption{Summary of results from experiments}
        \label{table:summary}
\end{table*}


In Table~\ref{table:summary} we summarize the obtained results, by splitting them into the different cases described above.
From this table we can see that our method usually finds a proper 
lumping (60.8\% of the cases). It is interesting to remark that in the remaining cases, for the majority of them (35.5\%) we obtain a non-proper lumping. These reductions are not covered by the theory of the current paper. Only the 3.7\% of cases were unsuccessful, either not findind any reduction of the models, or hitting the timeout of \timeout.

\begin{figure}[ht]
    \centering
    \includegraphics[width=\linewidth]{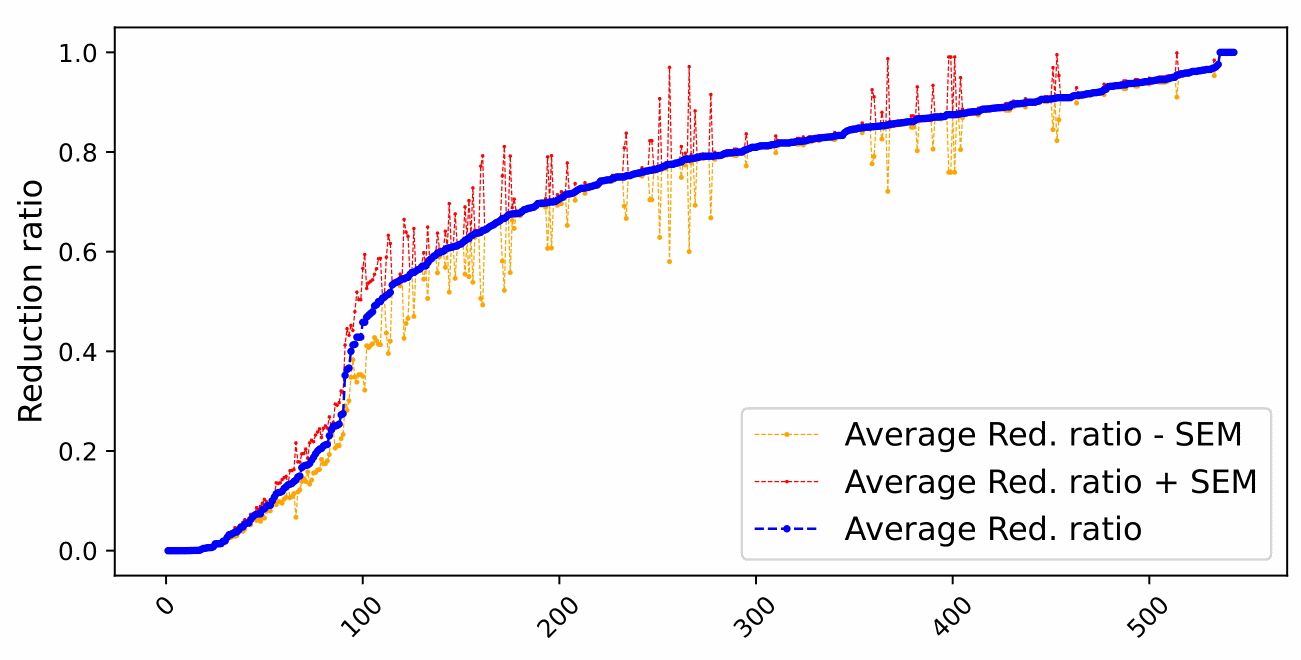}
    \caption{Average reduction ratio for models with some proper lumpings}
    \label{fig:results}
\end{figure}

In Figure~\ref{fig:results} we show, for each model, the average reduction ratio (i.e., number of reduced species over number of original ones in the system) obtained through all possible output matrices $O$. In this figure we have only considered the models where there is at least one case with a proper lumping, and removing all cases with general lumpings (i.e., with a reduction we can not interpret). We also considered the cases with no reduction. The standard error of the mean (SEM) is also included for each model.

From this figure we can see that for more than half of the networks (around 600 models) have at least one case with a proper 
lumping. Specifically, around 30\% of models have a reduction ratio of 65\% or less, thus providing a significant evidence of the usefulness of the proposed method.

\section{Conclusion}

In the paper we introduced constrained proper lumping, a model reduction technique for a class of bilinear control systems described by positive linear control systems (PCS) whose matrices take the role of control inputs. Our main result extends~\cite{uctmcTAC} and ensures that a so-called proper aggregation map $x \mapsto Lx$ is optimality-preserving if and only if the matrix constraints of the PCS $A(t) \in [\underline{A} ; \bar{A}]$ are such that $L$ is a bisimulation (aka lumping)~\cite{DBLP:journals/tac/PappasLS00,LI1994343} of the \emph{uncontrolled} dynamical systems $\partial_t x = \lb x$ and $\partial_t x = \ub x$. Building upon~\cite{10.1093/bioinformatics/btab258}, this allowed for the efficient computation of minimal proper lumpings with respect to given output maps~\cite{DBLP:journals/tac/PappasLS00,bisimulation_lin_sys_Schaft}. The practical applicability of the framework was demonstrated on a number of benchmark case studies. Future work will consider nonlinear reductions. 






\bibliographystyle{IEEEtran}
\bibliography{cdc2023}{}

\section*{Appendix}

\begin{proof}[Proof of Proposition \ref{lumping+:as:cdp}]
In the definition of $D$, for an index 
$i \in H_0$ we arbitrarily defined $\lambda_i = 1$. This choice ensures that $D$ is invertible but any other choice of $\lambda_i \neq 0$ would not affect the aggregation, as such indexes are the ones that do not belong to the support of any rows of $L$ (nor of $C$), and so they correspond to variables that do not appear in the aggregation $Lx$. 
The only not self-evident part of the statement is that $CD$ is a lumping of $A$ wrt $O$ if and only if $C$ is a lumping of $D A D^{-1}$ wrt $O D^{-1}$. It  follows from the definition of lumping and from the fact that $(C D)^+ = D^{-1} C^+$.
\end{proof}

\begin{proof}[Proof sketch of Proposition~\ref{generalized:UCRN}]
The direction from 1) to 3) follows by repeating the argumentation of~\cite[Proposition 4]{uctmcTAC}. Specifically, in the special case where $u(\cdot) = 0$, the argumentation carries over verbatim by noting that transient state probabilities $\pi_i$ can be replaced with non-negative state variables $x_i$, while lower and upper bounds can be given by Metzler rather than transition rate matrices. The general case where $u(\cdot) \in U$, instead, is then reduced to foregoing special case as follows. First, we pick an arbitrary $A(\cdot) \in [\lb; \ub]$ and $u(\cdot) \in U$. Using the solution $x^{(A,u)}(\cdot)$ and $A(\cdot)$, we construct $\B(\cdot) \in [\lbb; \ubb]$ as in the case of $u(\cdot) = 0$. Then, the proof of the special case ensures that $C x^{(A,u)} = y^{(R,u)}$. Conversely, for any $R(\cdot) \in [\lbb; \ubb]$ and $u(\cdot) \in U$, we construct $A(\cdot)$ from $R(\cdot)$ and $y^{(R,u)}(\cdot)$ as in the special case. Then, the proof ensures that $C x^{(A,u)} = y^{(R,u)}$. The direction 3) to 2) follows readily because $\Psi$ and $\Phi$ are $C$-invariant by assumption. This leaves us with the direction from $2)$ to $1)$. To this end, let us fix some $1 \leq r \leq k$ and consider the running cost $\Psi = 0$ and final cost $\Phi(x,u) = \sum_{j \in H_r} x_j + \| u \|_2$. Obviously, $\Phi$ and $\Psi$ are $C$-invariant.
Hence, $2)$ ensures that $V^{\inf} (x[0], \tau) = V^{\inf} (x[0]', \tau)$ and $V^{\sup} (x[0], \tau) = V^{\sup} (x[0]', \tau)$ where $x[0] = e_i$ and $x[0]' = e_{i'}$ for some $i,i' \in H_{r'}$ with $r \neq r'$ and $i \neq i'$, while $e_j$ denotes the $j$-th unit vector of $\RE^n$.
Similarly to~\cite[Proposition 5]{uctmcTAC}, we next observe that $x[0] = e_i$ yields
\[
\sum_{j \in H_r} x_j(t) = \Big[ \sum_{j \in H_r} A_{j,i}(0) + \sum_{j \in H_r} \sum_{\nu = 1}^m B_{j,\nu} u_\nu(0) \Big] t + o(t)
\]
because $A(t)$ and $u(t)$ are piecewise continuous with right limits. Since $u(\cdot) \geq 0$ and $B \in \REp^{n \times m}$, the infimum is attained for $u(0) = 0$ and $A(0) = \lb$, that is
\[
V^{\inf} (x[0], \tau) = \tau \sum_{j \in H_r} \lb_{j,i} + o(\tau).
\]
Since a similar reasoning applies to $V^{\inf} (x[0]', \tau)$, letting $\tau \to 0$ yields $\sum_{j \in H_r} \lb_{j,i} = \sum_{j \in H_r} \lb_{j,i'}$. Since the choice of $r \neq r'$, $H_r$ and $i,i' \in H_{r'}$ were arbitrary, this implies that $C$ is a forward equivalence~\cite{PNAScttv}, hence a lumping of $\lb$. The proof that $C$ is a lumping of $\ub$ proceeds in a similar fashion by considering the final cost $\Phi(x,u) = \sum_{j \in H_r} x_j - \eta \| u \|_2$. Here, $\eta \geq 0$ can be  chosen sufficiently large to ensure that $\sum_{j \in H_r} \sum_{\nu = 1}^m B_{j,\nu} u_\nu(0) \leq \eta \| u(0) \|_2$
because $U$ is bounded. With this, it follows that $V^{\sup} (x[0], \tau) = \tau \sum_{j \in H_r} \ub_{j,i} + o(\tau)$ and $V^{\sup} (x[0]', \tau) = \tau \sum_{j \in H_r} \ub_{j,i'} + o(\tau)$, allowing us to conclude that $\sum_{j \in H_r} \ub_{j,i} = \sum_{j \in H_r} \ub_{j,i'}$. Then again \cite{PNAScttv} yields that $C$ is a lumping of $\ub$.
%
%
%
\end{proof}

To increase readability, in the following we often drop the dependence on time $t$ from the functions.
\begin{lemma}\label{tech:lemma:conjugated:drifts}
Consider the system of ODEs $\partial_t x = Ax + Bu$ in $n$ variables, and let $D$ be an invertible $n \times n$ matrix. A function $x(t)$ satisfies $\partial_t x = Ax + Bu$ if and only if $z(t) : = D x(t)$ satisfies the system $\partial_t z = D A D^{-1} z + D B u$.
\end{lemma}
\begin{proof}
The function $z(t) = D x(t)$ satisfies $\partial_t z = D A D^{-1} z + D B u$ if and only if $\partial_t ( D x ) = D A D^{-1} D x + D B u$, if and only if $D \partial_t x = D A x + D B u$, if and only if $\partial_t x = A x + B u$.
\end{proof}

\begin{proof}[Proof of Theorem \ref{thm:coincidence:trajectories}]
We first prove the special case when $U = \{0\}$, and to increase readability we shall write $x^{(A)}$ and $J_A$ for $x^{(A,0)}$ and $J_{(A,0)}$, respectively.

By assumption, $L$ is a lumping of $\lb$ and $\ub$. Using Proposition \ref{lumping+:as:cdp}, write $L = C D$, where $D$ is a diagonal matrix with positive entries, 
and $C \in \{0,1\}^{k \times n}$ is a proper lumping of both $D \lb D^{-1}$ and $D \ub D^{-1}$.

Fix an initial condition $x[0] \in \R_{\geq 0}^n$, and let $\mathcal{S} = \{ L x^{(A)} \mid A \in [\lb;\ub] \}$. Then Lemma \ref{tech:lemma:conjugated:drifts} yields $D x^{(A)} = x^{ (D A D^{-1}) }$ for every $A \in [\lb;\ub]$, and so
\begin{align*}
\mathcal{S} & = \{ C D x^{(A)} \mid A \in [\lb;\ub] \} = \{ C x^{ (D A D^{-1}) } \mid A \in [\lb;\ub] \} \\
&= \{ C x^{ (D A D^{-1}) } \mid D A D^{-1} \in D [\lb;\ub] D^{-1} \}.
\end{align*}
As $D$ has positive entries, one can easily verify that $D \lb D^{-1} \leq D \ub D^{-1}$, and that the family of matrices $D [\lb;\ub] D^{-1}$ coincides with $[ D \lb D^{-1}; D \ub D^{-1}]$. Then
\begin{align*}
\mathcal{S} &= \{ C x^{ (D A D^{-1}) } \mid D A D^{-1} \in [ D \lb D^{-1}; D \ub D^{-1}] \} \\
&= \{ C x^{ (M) } \mid M \in [ D \lb D^{-1}; D \ub D^{-1}] \}.
\end{align*}
As $C$ is a proper lumping of the bounds in the above set, we can apply Proposition~\ref{generalized:UCRN} to obtain that
\begin{align*}
\mathcal{S} & = \{ y^{ ( \B ) } \mid \B \in [ C ( D \lb D^{-1} ) C^+; C ( D \ub D^{-1} ) C^+ ] \} \\
& = \{ y^{ ( \B ) } \mid \B \in [ L \lb L^+; L \ub L^+ ] \},
\end{align*}
where in the last step we used the fact that $L^+ = (C D)^+ = D^{-1} C^+$. This proved the first part of the statement.

Now from it we deduce that $L$ preserves the costs of the PCS, so fix an initial condition $x[0]\in \R^n_{\geq 0}$, a time $\tau \in \R$, and $\fin, \run$ final and running costs that are $L$-invariant. As here we are assuming $U = \{0\}$, we drop the dependence of $\fin, \run$ on $u\in U$.
For every $A\in [\lb;\ub]$, let $\B \in [ L \lb L^+; L \ub L^+ ]$ be the matrix provided by the first part of the proof, 
and satisfying $L x^{(A)} = y^{ (\B) }$.
Then
$
\fin \big( x^{(A)} (\tau) \big) = \hfin\big( L{ x^{(A)} (\tau) } \big) = \hfin\big( y^{ (\B) } (\tau) \big),
$
and $\run\big(t, x^{(A)} (t) \big) = \hrun \big( t, L{x^{(A)} (t) } \big) = \hrun \big( t, y^{ (\B) } (t) \big)$ for every $t \geq 0$, giving $J_{A}(x[0], \tau) = J_{ \B } (Lx[0], \tau)$.
The same argument as above shows that for every $\B \in [ L \lb L^+; L \ub L^+ ]$ there is $A\in [\lb;\ub]$ such that
$J_{A}(x[0], \tau) = J_{ \B } (Lx[0], \tau)$, so the final coincidence of the optimal values 
is then obvious. 

We now deal with the general case, so we drop the assumption of having 
$U = \{ 0 \}$. 
Pick arbitrary $A(\cdot) \in [\lb; \ub]$ and $u(\cdot) \in U$. Using the solution $x^{(A,u)}(\cdot)$ and $A(\cdot)$, we construct $\B(\cdot) \in [\lbb; \ubb]$ as in the case of $U = \{0\}$. Then, the proof of the special case ensures that $L x^{(A,u)} = y^{(R,u)}$. Conversely, for any $R(\cdot) \in [\lbb; \ubb]$ and $u(\cdot) \in U$, we construct $A(\cdot)$ from $R(\cdot)$ and $y^{(R,u)}(\cdot)$, as in the shown special case. Then, the proof ensures that $L x^{(A,u)} = y^{(R,u)}$. From this, the coincidence of the optimal values follows immediately.
%
%
\end{proof}

\begin{proof}[Proof of Proposition~\ref{prop:reconstruction:general}]
As $L^+ = (C D)^+ = D^{-1} C^+$, we have $\B(\cdot) \in [ L \lb L^+; L \ub L^+ ] = [ C \underline{M} C^+; C \overline{M} C^+ ]$. As $C$ is a lumping of both $\underline{M}$ and $\overline{M}$, we can apply Proposition \ref{generalized:UCRN}, and let $\tilde{A}(\cdot) \in [ \underline{M}; \overline{M} ]$ be the matrix defined in item 3ii), thus satisfying $C x^{( \tilde{A},u )} = y^{(\B,u)}$.
It is easy to check that for every $i,j$ we have $A_{i,j}(\cdot) = \lambda_{i}^{-1} \lambda_{j} \tilde{A}_{i,j}(\cdot) = ( D^{-1} \tilde{A}(\cdot) D )_{i,j}$, i.e. $A(\cdot)$ in our statement satisfies $A(\cdot) = D^{-1} \tilde{A}(\cdot) D$.
By construction, $\tilde{A}(\cdot) \in [ \underline{M}; \overline{M} ] = [D \lb D^{-1}; D \ub D^{-1}] = D [\lb; \ub] D^{-1}$, so $A(\cdot) = D^{-1} \tilde{A}(\cdot) D \in [\lb; \ub]$. The final part is easy to check, and Lemma \ref{tech:lemma:conjugated:drifts} yields $L x^{(A,u)} = C D x^{(A,u)} = C x^{( D A D^{-1},u )} = C x^{( \tilde{A},u )} = y^{(\B,u)}$.
\end{proof}

\begin{proof}[Proof of Theorem \ref{thm:coincidence:cost:functionals}]
We have to prove that $L$ is a proper lumping of both $\lb$ and $\ub$. Define $C$ and $D$ as in Proposition \ref{lumping+:as:cdp}, so that $L = CD$ and $\rowspan(O D^{-1}) \subseteq \rowspan (C)$; now it is sufficient to prove that $C$ is a proper lumping of both $D \lb D^{-1}$ and $D \ub D^{-1}$. In view of Proposition~\ref{generalized:UCRN}, we check that $C$ preserves the costs of the PCS $([ D \lb D^{-1} ; D \ub D^{-1}], O D^{-1}, D B, U)$. This means checking that the latter PCS and its reduction by $C$, that is
$([ L \lb L^+; L \ub L^+ ], O L^+, L B, U)$,
have the same optimal values (to write the above reduced PCS we used that $L^+ = D^{-1} C^+$). 
To this end, fix a time $\tau \in \R$, an arbitrary $z[0]\in \R^n_{\geq 0}$, and arbitrary final and running costs $\phi$, $\psi$ that are $C$-invariant. First, consider the new final cost function $\fin \colon \R^n \times \R^m \to \R$ defined by $\fin ( x, u) = \phi ( D x, u) $ for $x \in \R^n$ and $u \in \R^m$. It is easy to check that $\fin$ is $CD$-invariant.
Similarly, we denote by $\run$ the new running cost function $\R \times \R^n \times \R^m\to \R$ defined by $\run(t,x,u) = \psi(t, Dx,u)$, and it holds that also $\run$ is $L$-invariant; 
moreover, it holds that
\begin{equation}\label{hat:costs:circ:D}
\hat{\phi} = \hfin \text{ and } \hat{\psi} = \hrun.
\end{equation}

Fix $M \in [ D \lb D^{-1} ; D \ub D^{-1}]$ and $u \in U$. Let $A \in [ \lb; \ub]$ be such that $M = D A D^{-1}$.
Using the initial condition $x[0] = D^{-1} z[0] \in \R^n_{\geq 0}$, time $\tau$, and the costs $\fin$, $\run$, apply the assumption on $L$ to preserve the costs. This way, we get $\B \in [ L \lb L^+; L \ub L^+]$ and $v \in U$ such that
\begin{equation}\label{eq:dom2}
J_{(A,u)} (x[0], \tau) = J_{(\B,v)} (Lx[0], \tau).
\end{equation}

Our goal now is to prove that
\begin{equation}\label{eq:dom}
J_{(M,u)} (z[0], \tau) = J_{(\B,v)} (Cz[0], \tau);
\end{equation}
to deduce \eqref{eq:dom} from \eqref{eq:dom2}, we show that both their left-hand sides and right-hand sides coincide. We begin by studying their left-hand sides. Let us denote by $z$ the solution of $\partial_t z = M z + D B u$, with initial condition $z[0]$ (and along which the functional $J_{(M,u)} (z[0], \tau)$ is computed, by definition).
By Lemma \ref{tech:lemma:conjugated:drifts}, we have $z = D x^{(A,u)}$, and then 
$\fin ( x^{(A,u)}(\tau), u(\tau) ) = \phi ( D x^{(A,u)}(\tau), u(\tau) ) = \phi ( z(\tau), u(\tau) )$. Similarly, $\run (t, x^{(A,u)}(t), u(t) ) = \psi (t, z(t), u(t) )$ for every $t \geq 0$, giving $J_{(A,u)} (x[0], \tau) = J_{(M,u)}(z[0], \tau)$.
Now we compare the functionals in the right-hand sides of \eqref{eq:dom2} and \eqref{eq:dom}.
First, both of these functionals are computed along the same trajectory $y^{(\B,v)}$, solution of the problem $\partial_t y = \B y + L B u$, and with the same initial condition $Lx[0] = CDx[0] = Cz[0]$. The evaluation of $J_{(\B,v)} (Lx[0], \tau)$ involves the reduced costs $\hfin$ and $\hrun$, while $J_{(\B,v)} (Cz[0], \tau)$ involves the reduced costs $\hat{\phi}$ and $\hat{\psi}$, so we conclude by applying \eqref{hat:costs:circ:D}. With a similar argument, one can see that also for every $z[0]\in \R^n_{\geq 0}$, and $\B \in [ L \lb L^+; L \ub L^+]$, and $v \in U$, there are $M \in [ D \lb D^{-1} ; D \ub D^{-1}]$ and $u \in U$ such that \eqref{eq:dom} holds for all costs that are $C$-invariant, so to finally establish that
$C$ preserves the costs of the PCS $([ D \lb D^{-1} ; D \ub D^{-1}], O D^{-1}, D B, U)$.
\end{proof}

\subsection{Proof of Theorem~\ref{Clue:for:m:and:M}}\label{sec:proof:existence:RWE}

To prove Theorem~\ref{Clue:for:m:and:M} 
we need some technical results first. 
For $L, \rwlL\in\R^{k \times n}$ , the property $\rowspan(L) = \rowspan(\rwlL)$ 
means that the rows of $L$ and those of $\rwlL$ span the same vector space, i.e., there is an invertible matrix $S \in \R^{k\times k}$ such that $L = S\rwlL$.

Let us denote by $e_{\w}$ the $\w$-th unit column vector in $\R^k$, having all $0$ entries but the $\w$-th one, which is $1$.

\begin{lemma}[Disjoint supports]\label{new:supp(ra):technical}
    Let $L,\rwlL \in \R^{k\times n}$ with rank $k$ be such that $L = S\rwlL$, for an invertible matrix $S$. 
    If the rows of $\rwlL$ have pairwise disjoint supports, then the columns of
    $L$ are multiples of the columns of $S$.
\end{lemma}
\begin{proof}
    Since $L = S\rwlL$, we have that the $i$-th column $L_i$ of $L$ can be written as $L_i = S\rwlL_i$. Since the rows of $\rwlL$ have disjoint
    supports, there is at most one $\w$ in $\{1,\ldots, k\}$ such that $\rwlL_{\w,i} \neq 0$. If there is no such $\w$, then $0 = \rwlL_i = L_i$ and we are done. Otherwise, $\rwlL_i = \rwlL_{\w,i} e_\w$, and so $L_i = S \rwlL_{\w,i} e_{\w} = \rwlL_{\w,i} S_{\w}$.
\end{proof}

This result is enough to characterize when, given a lumping $L \in \R^{k\times n}$, there is another lumping $\rwlL \in \R^{k\times n}$ with $\rowspan(\rwlL) = \rowspan(L)$ whose rows have disjoint supports.

\begin{definition}[Equivalence over columns]\label{def:tilde:main}
    Let $L\in \R^{k\times n}$ have rank $k\leq n$. Let $\I = \{1,2,\ldots,n\} \setminus H_0 = \{ 1 \leq i \leq n \mid L_i \neq 0 \}$ (i.e., the set of indices of the columns that are
    not identically zero) and write $i \sim j$ for any $i,j \in \I$ whenever the columns $L_i$, $L_j$ are multiples of each other.
\end{definition}

\begin{proposition}\label{thm:existence:lump+}
    Let $L$, $\I$ be as in Definition~\ref{def:tilde:main}. Then ${|\I/\!\!\sim\!| = k}$ if and only if there is $\rwlL\in \R^{k\times n}$ with ${\rowspan(\rwlL) = \rowspan(L)}$ whose rows have pairwise disjoint
    supports.
\end{proposition}
\begin{proof}
    Assume such $\rwlL$ exists. Since $\rowspan(\rwlL) = \rowspan(L)$, there is an invertible $S$ with $L = S\rwlL$. By Lemma~\ref{new:supp(ra):technical},
    the columns of $L$ are multiples of the columns of $S$. Hence, $|\I/\!\!\sim\!| \leq k$. On the other hand, since $\rwlL$ has rank $k$, all its rows have at least
    one non-zero element. Thus, for every column of $S$ there is a column of $L$ that is a multiple of its, meaning that $|\I/\!\!\sim\!| \geq k$.     Now, assume that $|\I/\!\!\sim\!| = k$ and let $\calH = \{H_1,\ldots, H_k\}$ be the partition of $\I$ obtained by $\sim$. For each $\w = 1,\ldots, k$,
    we fix one representative $i(\w) \in H_{\w}$. For $i,j \in H \in \calH$, we denote by $q(i,j)$ the ratio between the $i$-th and $j$-th column of $L$ (note that neither $L_i$ nor $L_j$ is the zero column). 
    Set $\rwlL$ as follows: $\rwlL_{\w,j} = q_{i(\w), j} $ if $j \in H_{\w}$, and zero otherwise. By construction, the rows of $\rwlL$ have disjoint supports (since the support of the $\w$-th row of $\rwlL$ is $H_{\w}$). Moreover, if we
    define $S = (L_{i(1)} | \ldots | L_{i(k)})^T$ we have $L = S\rwlL$, so ${\rowspan(L) = \rowspan(\rwlL)}$.
\end{proof}


\begin{proof}[Proof of Theorem \ref{Clue:for:m:and:M}]
    Assume that there is a constrained proper lumping $\rwlL \in \RE_{\geq 0}^{k \times n}$ of the PCS such that 
    $\rowspan(\rwlL)$ has dimension $k$.
    By construction, $L$ from Proposition \ref{prop:alg:clue} has the property that $\rowspan(L)$ has also dimension $k$,
    and it is minimal satisfying those conditions, so we have $\rowspan(L) = \rowspan(\rwlL)$. To prove that $L$ is a constrained proper lumping, we first show that its rows have disjoint supports. Consider $\I$ as in Definition~\ref{def:tilde:main} for
    the lumping $L$. By Proposition~\ref{thm:existence:lump+} we get $|\I/\!\!\sim\!| = k$. By Proposition~\ref{prop:alg:clue},
    $L$ is in reduced echelon-form and so the $k$ vectors in $\{e_{\w} \mid 1 \leq \w \leq k \}$ are columns of $L$. We can take them as representatives
    of the equivalence classes of~$\sim$.
    Then for every non-zero column $L_i$ of $L$, there is $\w \in \{1,\ldots,k\}$ with $L_i \sim e_{\w}$, and in particular $L_i$ only
    has a non-zero entry in the $\w$-th row. Thus, the rows of $L$ have pairwise disjoint supports. Next, we show that $L \in \RE_{\geq 0}^{k \times n}$. To this end, assume by contradiction that $L$ has a negative entry, say $L_{\w,i} < 0$.
    From $\rowspan(L) = \rowspan(\rwlL)$, there is an invertible matrix $S$ with $\rwlL = SL$. Since the rows of $L$ have pairwise disjoint supports,
    we can apply Lemma~\ref{new:supp(ra):technical} in this setting. As the column $L_i$ is not zero,
    the proof of Lemma~\ref{new:supp(ra):technical} shows that $\rwlL_i = L_{\w,i} S_j$.
On the other hand, $L$ is in reduced echelon-form so it has a column, say $L_j$, such that $L_j = e_{\w}$. Then $\rwlL = SL$ gives $\rwlL_j = S L_j = S e_{\w} = S_{\w}$.
    So we have deduced that $\rwlL$ has as columns the vectors $S_{\w}$ and $L_{\w,i} S_j$.
    Since $\rwlL$ is a proper lumping we obtain that $S_{\w} = 0$, but this contradicts the hypothesis of $S$ being invertible.
\end{proof}

\end{document}